\documentclass{amsart}
\usepackage{amsmath}
\usepackage{amssymb}
\usepackage{color}
\usepackage{enumerate}
\usepackage{xy}
\usepackage{xypic}
\usepackage{url}
\usepackage[normalem]{ulem}

\newcommand{\bB}{\mathbf{B}}
\newcommand{\bG}{\mathbf{G}}
\newcommand{\chD}{\check{D}}
\newcommand{\cF}{\mathcal{F}}

\newcommand{\CC}{\mathbb{C}}
\newcommand{\QQ}{\mathbb{Q}}
\newcommand{\RR}{\mathbb{R}}
\newcommand{\ZZ}{\mathbb{Z}}

\newcommand{\sS}{\mathsf{S}}

\newcommand{\dqs}{\mathsf{D}\mathsf{Q}\mathsf{S}}
\newcommand{\dcqs}{\mathsf{D}\mathbb{C}\mathsf{Q}\mathsf{S}}

\newtheorem{thm}{Theorem}[section]
\newtheorem{proposition}[thm]{Proposition}
\newtheorem{claim}{Claim}[thm]
\newtheorem{lem}[thm]{Lemma}
\newtheorem{convention}[thm]{Convention}

\theoremstyle{definition}
\newtheorem{definition}[thm]{Definition}

\theoremstyle{remark}
\newtheorem{remark}[thm]{Remark}

\newenvironment{claimproof}{\vspace{.1in} \noindent {\bf Proof of Claim:}}{\hspace{\fill} $\maltese$ \vspace{.2in}}

\newcommand{\trdeg}[2]{\operatorname{tr.deg}_{#1}(#2)}

\begin{document}
	\title{Likely intersections}
	\author{Sebastian Eterovi\'{c}}
 \address{School of Mathematics \\ University of Leeds \\ Leeds, LS2 9JT \\ United Kingdom}
\email{s.eterovic@leeds.ac.uk} 
 \author{Thomas Scanlon}
 \address{Department of Mathematics \\ University of California, Berkeley \\ Evans Hall \\
 Berkeley, CA 94720-3840 \\ United States of America}
 \email{scanlon@math.berkeley.edu}
\begin{abstract}
We prove a general likely intersections 
theorem, a counterpart to the Zilber-Pink 
conjectures, 
under the assumption that the Ax-Schanuel 
property and some mild additional conditions
are known to hold for a given 
category of complex quotient spaces definable
in some fixed o-minimal expansion of the 
ordered field of real numbers.  

For an instance of our general result, consider the case of 
subvarieties of Shimura varieties.  Let 
 $S$ be a Shimura variety.  Let $\pi:D \to \Gamma \backslash D = S$ realize
$S$ as a quotient of 
$D$, a homogeneous space for the action of 
a real algebraic group $G$, by the action of $\Gamma < G$, an arithmetic subgroup.
Let  $S' \subseteq S$ be a special subvariety of $S$ realized as $\pi(D')$ 
for $D' \subseteq D$ a homogeneous space for an algebraic subgroup of $G$.
Let $X \subseteq S$ be an irreducible subvariety of $S$ not contained 
in any proper weakly special subvariety of $S$.  Assume that the 
intersection of $X$ with $S'$ is persistently likely meaning that whenever 
$\zeta:S_1 \to S$ and $\xi:S_1 \to S_2$ are maps of Shimura varieties 
(meaning regular maps of varieties induced by maps of the corresponding 
Shimura data) with $\zeta$ finite,  $\dim \xi \zeta^{-1} X + \dim \xi \zeta^{-1} S' 
\geq \dim \xi S_1$.   Then $X \cap \bigcup_{g \in G, \pi(g D') \text{ is special }} \pi(d D')$ 
is dense in $X$ for the Euclidean topology.
\end{abstract}

\maketitle

\section{Introduction}

If $X$ is an complex 
manifold and $f:Y \to X$ and $g:Z \to X$
are two  sufficiently nice maps from complex analytic spaces to 
$X$, then we say that an intersection 
between $Y$ and $Z$ is \emph{unlikely} if 
$\dim f(Y) + \dim g(Z) < \dim X$ \footnote{ We will be working with cases in which $f(Y)$ is 
automatically a finite Boolean combination of (at least) real 
analytic subvarieties of $X$.  This
happens, for example, when $f:Y \to X$ is proper or definable in an 
o-minimal expansion of the real field. 
 In this case, we could take $\dim$ to mean the o-minimal dimension or 
 the maximal dimension of a submanifold of $f(Y)$.} and that it is \emph{likely} otherwise.
Conjectures of Zilber-Pink type predict that in many cases of 
interest there are very few unlikely intersections.  
For example,  the Zilber-Pink 
conjecture for Shimura varieties takes the 
following form. We let 
$X$ be a Shimura variety and $Y \subseteq X$
 a Hodge generic subvariety, meaning that there is no proper 
weakly special subvariety (in the sense of Shimura varieties) 
$T \subsetneq X$ with $Y \subseteq T$.  The 
conclusion of the conjecture in this 
case is that
the union of all unlikely intersections 
between $Y$ and special subvarieties of $X$ is not Zariski dense in $Y$.   
In this paper, we address the complementary 
question of 
describing the likely intersections.

This is not the first time that the likely 
intersection problem has been addressed.  
In Section~\ref{sec:applications} we discuss
several specializations of our main theorem 
and the relation between our results and those
which appear in the literature. Three notable
instances are the ``all or
nothing'' theorem  of Baldi, Klingler, and Ullmo~\cite{BKU}
on the density of the typical Hodge locus,  {the results of Tayou and Tholozan~\cite{tayou-tholozan} 
describing the typical Hodge locus of polarized variations
of Hodge structure over a smooth complex quasiprojective variety, and 
the work of Gao~\cite{Gao-rank-Betti,Gao-rank-Betti-er} 
(following on the work of Andr\'{e}, 
Corvaja, and Zannier~\cite{ACZ}) on the generic rank of the 
Betti map, from which a sufficient condition for the density 
of torsion in subvarieties of abelian schemes is derived.

We formulate and prove our main theorem (Theorem \ref{thm:density-special}) in terms of 
 definable analytic maps from complex algebraic varieties to definable 
complex quotient spaces. See Section~\ref{sec:cqs} for details.
In brief, a definable complex quotient space $S$ is a 
complex analytic space which may be presented as a double
coset space $\Gamma \backslash G / M$ where $G$ is the connected
component of the real points of an algebraic group, $M \leq G$ is a 
suitable subgroup, and $\Gamma \leq G$ is a discrete subgroup
together with a choice of a definable (in a fixed
o-minimal expansion of the real numbers) fundamental 
set $\mathcal{F}$.  Examples of such definable 
complex quotient spaces include complex tori, Shimura 
varieties, Hopf manifolds, and (mixed) period spaces.  We then 
define a special subvariety of a definable 
complex quotient space $S$ to be the image of a map 
of definable complex quotient space $S' \to S$ (or 
possibly taken only from some subcategory of 
definable complex quotient spaces) which is a 
definable map induced by an algebraic group homomorphism 
followed by translation.  For example, the special subvarieties
of complex tori would be translates by torsion points
of subtori and the special subvarieties of period spaces 
would come from period subdomains.  We must also consider a more
general class of weakly special varieties, which are 
the fibers of maps of definable complex quotient spaces and 
the images of these fibers under other maps of 
definable complex quotient spaces.  We identify a 
condition we call \emph{well-parameterization of weakly 
special subvarieties} whereby all of the weakly special varieties in 
some given definable complex quotient space come from 
those appearing in countably many families of weakly 
special subvarieties. In our applications this condition is easy to 
verify as we restrict to subcategories of definable complex 
quotient spaces for which there are only countably many 
morphisms all told. 

Consider $S = \Gamma \backslash G / M = \Gamma \backslash D$ and 
$D' \subseteq D$ a homogeneous space for a subgroup $G' \leq G$.
We will write $\pi_\Gamma:D \to S$ for the quotient map. 
It may happen that $\pi_\Gamma (gD') \subseteq S$ 
is a special subvariety of $S$ for many choices of $g \in G$, where
``many'' might mean that the set of such $g$ is dense in $G$.
For example, if $D = \mathfrak{h}^2$ is the Cartesian 
square of the upper half plane, 
$\Gamma = \operatorname{PGL}_2^2(\ZZ)$ and 
$D' = \{ (\tau,\tau) : \tau \in 
\mathfrak{h} \}$ is the diagonal, then all modular plane curves
may be expressed as $\pi_\Gamma(gD')$ as $g$ ranges through 
$\operatorname{PGL}_2^2(\mathbb{Q})$ acting via pairs of 
rational linear transformations.   In 
such a situation we might expect that if 
$X$ is a quasiprojective complex algebraic variety and 
$f:X^\text{an} \to S$ is a definable complex analytic map from the 
analytification of $X$ to $S$, then the set $f^{-1} \bigcup_{g \in G, 
\pi_\Gamma(gD') \text{ special }} \pi_\Gamma(g D')$ 
of special intersections is dense in 
$X$ provided that the intersections are likely in the sense that 
$\dim f(X) + \dim D' \geq \dim S$.  This is not quite right as the 
intersection may become unlikely after transformation through a special 
correspondence. We account for this complication with the notion of 
\emph{persistently likely intersections}.   Our main 
Theorem~\ref{thm:density-special} asserts that if we know that the
Ax-Schanuel theorem holds for our given category of definable complex 
quotient spaces (which satisfies some natural closure properties
and the well-parameterization of weakly special subvarieties condition), then,
in fact, when the intersection of $f(X)$ with $\pi_\Gamma(D')$ 
is persistently likely, the set of special intersections is 
dense in $X$.

When working in the setting of variations of Hodge structures,
our result improves the ``all or nothing'' theorem of \cite{BKU} first by
 giving a criterion for the typical Hodge locus to be dense (namely \emph{persistent likeliness}),
and secondly by obtaining a euclidean dense set of likely intersections
using a specific subcollection of the family of all special varieties.
On the other hand, the general setting
of definable quotients spaces does not allow for measure theoretic techniques, 
so we are unable to recover the equidistribution results appearing
in, for example, \cite{BKU} and \cite{tayou-tholozan}. But our result does
imply some of the applications of these equidistribution results. 
For example, in the setting of the moduli space of principally polarized
abelian varieties of dimension $g$, given two subvarieties $S$ and $D$ of $\mathcal{A}_g$
of complimentary dimensions and so that $\dim S\leq 2$, 
our result implies the part of \cite[Theorem 1.22]{tayou-tholozan}
stating that the set of points in $S$ isogenous to a 
point in $D$ is euclidean dense in $S$. 
See also Section \ref{sec:applications} for more applications. 
 An advantage of the present work is that our results 
apply directly to the likely intersections problem to contexts not 
covered by~\cite{tayou-tholozan}, such as for mixed Shimura varieties, or more 
generally, for variations of mixed Hodge structures.

Unsurprisingly, Ax-Schanuel theorems play key roles in the existing 
proofs of the density of special intersections.  What may be surprising
is that our argument is not an abstraction of the proofs appearing,
for example, in~\cite{BKU}, \cite{tayou-tholozan} or~\cite{Gao-rank-Betti}.  
Instead, our arguments are inspired by the work of Aslanyan and 
Kirby~\cite{AK-blur}, especially with the proof of their Theorem 3.1.
The reader will recognize a resemblance between our notion of 
persistent likeliness and the $J$-broad and $J$-free conditions
of~\cite{AK-blur}, which themselves extend freeness and rotundity
conditions from earlier works on existential closedness as a 
converse to Schanuel-type statements.  While the contents 
of our arguments differ, the structure of many of the results of 
Daw and Ren in~\cite{DaRe} inspired our approach.

Our own interest in the likely intersections problem was motivated 
by~\cite[Conjecture 3.13]{PiSc} 
in the second author's work with Pila on effective 
versions of the Zilber-Pink conjecture.

This paper is organized as follows.  In Section~\ref{sec:cqs} we 
define definable quotient spaces and develop some of their basic 
theory.  These definitions and results owe their form to the formalism 
of Bakker, Brunebarbe, Klingler and Tsimerman~\cite{BKT, BBKT} used
to study arithmetic quotients and more generally mixed period spaces.  In 
that section we express precisely what the Ax-Schanuel condition 
means and show that under the well-parameterization of 
weakly special subvarieties hypothesis, it implies a uniform version of itself.  
Section~\ref{sec:density} is devoted to the statement and 
proof of our main theorem on dense special intersections.  
In Section~\ref{sec:applications} we detail several specializations
of the main theorem, including to Hodge loci, intersections with modular
varieties, and density of torsion in subvarieties of abelian schemes.

\subsection*{Acknowledgments}  During the writing of this paper S.E. 
was partially supported through the NSF grant RTG DMS-1646385 and
T.S. was
partially supported by NSF grants DMS-1800492, FRG DMS-1760414, and 
DMS-22010405.  The authors have no competing interests, financial or
otherwise. S.E. thanks E. Ullmo and G. Baldi for hosting him at 
IH\'{E}S, for sharing an early version 
of~\cite{BKU}, and for discussing problems around the density of special 
intersections.  
T.S. thanks Z. Gao, M. Orr, and U. Zannier for offering detailed accounts
of the state of the art on likely intersection problems for modular 
curves and torsion in abelian schemes.  Both authors thank an anonymous 
referee of an earlier version of this paper for suggesting several improvements.

\section{Complex quotient spaces and $\mathsf{S}$-special 
varieties}
\label{sec:cqs}

We express our theorems on likely intersections in terms 
of classes of definable complex quotient spaces.    
Our formalism is similar to what appears in~\cite{BKT}, though
we explicitly include the fundamental domain giving the definable
structure as part of our data.

Throughout we work in an appropriate o-minimal expansion of the real field $\mathbb{R}$ (usually $\mathbb{R}_{\mathrm{an},\exp}$), and the word \emph{definable} is meant with respect to this choice of o-minimal structure. 

\begin{definition}
\label{def:definable-quotient-space}
A \emph{definable quotient space} is given by the data of 
\begin{itemize}
\item  a definable group $G$,
\item  a definable compact subgroup $M \leq G$ of $G$,
\item a discrete subgroup $\Gamma \leq G$ of $G$, and
\item $\mathcal{F} \subseteq D := G/M$ a definable open fundamental set 
for the action of $\Gamma$ on $D$ 
(that is, $D = \bigcup_{\gamma \in \Gamma} \gamma \mathcal{F}$ and there is a
finite subset $\Gamma' \subseteq \Gamma$
so that if $x \in \mathcal{F}$ and $\gamma x \in \mathcal{F}$ for some $\gamma \in \Gamma$, then $\gamma \in \Gamma'$) for which 
the closure $\overline{\mathcal{F}}$ of 
$\mathcal{F}$ is contained in 
$\bigcup_{\gamma \in \Gamma''} \gamma \mathcal{F}$ for some finite subset 
$\Gamma'' \subseteq \Gamma$.
\end{itemize}

We write $S_{\Gamma,G,M;\mathcal{F}}$ 
both for the quotient space $\Gamma \backslash D = \Gamma \backslash G / M$ regarded
as a definable, real analytic space where the definable structure comes from $\mathcal{F}$, and 
for the data $(G,M,\Gamma,\cF)$ giving this space. We denote the corresponding quotient map by $\pi_{\Gamma}:D\rightarrow S_{\Gamma,G,M;\mathcal{F}}$.  When the data are understood, 
we suppress them and write $S$ for $S_{\Gamma,G,M;\mathcal{F}}$ and 
$\pi:D \to S$ for the quotient map.
\end{definition}

The class of definable quotient spaces forms a category $\dqs$ with the following notion of a 
morphism.

\begin{definition}
\label{def:morphism-dqs}  A \emph{morphism} $f:S_{\Gamma_1,G_1,M_1;\cF_1} \to S_{\Gamma_2,G_2,M_2;\cF_2}$
is given by a definable map of groups $\varphi:G_1 \to G_2$ and an element $a \in G_2$ for which
\begin{itemize}
    \item $\varphi(M_1) \leq M_2$,
    \item $\varphi(\Gamma_1) \leq a^{-1} \Gamma_2 a$, and
    \item there is a finite 
set $\Xi \subseteq \Gamma_2$ with $a \overline{\varphi}(\cF_1) \subseteq \bigcup_{\xi \in \Xi} \xi \cF_2$, 
where $\overline{\varphi}:D_1 \to D_2$ is the map induced on the quotient space.
\end{itemize}
The induced map 
on the double quotient space is a definable real analytic map which we also denote by $f$.
\end{definition}

\begin{remark}
If $f:S_{\Gamma_1,G_1,M_1;\cF_1} \to S_{\Gamma_2,G_2,M_2;\cF_2}$ is a morphism
of definable quotient spaces as in Definition~\ref{def:morphism-dqs},
then it is definable in the sense that 
its graph on the specified fundamental domains
\[\operatorname{Graph}(f) := \{ (x,y) \in \mathcal{F}_1 \times \mathcal{F}_2 :
f(\pi_{\Gamma_1}(x)) = \pi_{\Gamma_2}(y) \} \]
is definable.  Indeed, it is clear that the induced map 
$\overline{\varphi}:D_1 \to D_2$ is definable and for $(x,y) 
\in \mathcal{F}_1 \times \mathcal{F}_2$ we have 
that $(x,y) \in \operatorname{Graph}(f)$ if and only if 
$\bigvee_{\xi \in \Xi} \xi \cdot y = a \overline{\varphi} (x)$.
\end{remark}

The category $\dqs$ has a terminal object and
is closed under fiber products.

\begin{proposition}
\label{prop:terminal} 
There is a terminal object in $\dqs$.
\end{proposition}
\begin{proof}
Take $G = M = \Gamma = \{ 1 \}$ to be 
trivial group and then let $\mathcal{F} := 
G/M$, which is a singleton.   The one point 
space $\{ \ast \} = S_{G,\Gamma,M;\mathcal{F}}$
is a definable quotient space and for any 
definable quotient space 
$S = S_{G',\Gamma',M';\mathcal{F}'}$, the 
unique set theoretic map $S \to \{ \ast \}$ 
is induced by the unique map of groups 
$\varphi:G' \to \{ 1 \}$ and $a = 1 \in \{ 1 \}$.
\end{proof}

\begin{proposition}
\label{prop:fiber-product} 
The category $\dqs$ has pullbacks.
\end{proposition}
\begin{proof}
More concretely, 
we need to show that 
given maps $f:S_2 \to S_1$ and $g:S_3 \to S_1$
of definable quotient spaces, there 
are maps of definable quotient 
spaces 
$\overline{g}:S_4 \to S_2$ and 
$\overline{f}:S_4 \to S_3$ fitting into the 
following Cartesian square.

\[ \xymatrix{S_4 \ar[d]_{\overline{g}} \ar[r]^{\overline{f}} & S_3 \ar[d]^{g} \\
S_2 \ar[r]^{f} & S_1 } \]

Express $S_i = S_{G_i,\Gamma_i,M_i;\mathcal{F}_i}$ and $f:S_2 \to S_1$ as 
$\Gamma_2 x M_2 \mapsto \Gamma_1 a_2 \varphi(x) M_1$ and $g:S_3 \to S_1$ 
as $\Gamma_3 x M_3 \mapsto \Gamma_1 a_3 \psi(x) M_1$ where 
$\varphi:G_2 \to G_1$ and $\psi:G_3 \to G_1$ are definable group homomorphisms,
$a_2 \in G_1$ and $a_3 \in G_1$.   Let $G_4 := G_2 \times_{G_1} G_3$, 
$M_4 := M_2 \times_{M_1} M_3$, and $\Gamma_4 := \Gamma_2 \times_{\Gamma_1} \Gamma_3$. 
For the fiber products defining $G_4$ and $M_4$ we use the maps $\varphi:G_2 \to G_1$ (and 
its restriction to $M_2$) and $\psi:G_3 \to G_1$ (and its restriction to $M_3$).  
In defining $\Gamma_4$, we use the maps $x \mapsto a_2 \varphi(x) a_2^{-1}$ and 
$x \mapsto a_3 \psi(x) a_3^{-1}$.    Regarding $G_4$ as a subgroup of $G_2 \times G_3$ and 
then $D_4 := G_4/M_4$ as a subset of $G_2/M_2 \times G_3/M_3 = D_2 \times D_3$, we let 
$\mathcal{F}_4 := D_4 \cap (\mathcal{F}_2 \times \mathcal{F}_3)$.  
\end{proof}

A useful observation is that 
morphisms of definable 
quotient spaces always factor as a surjective map followed by a map induced by an inclusion of subgroups.

\begin{proposition}
\label{prop:factorization} 
Every map $f:S_1 \to S_2$ of definable quotient 
spaces fits into a commutative diagram
\[ \xymatrix{S_1  \ar[rr]^f \ar[rd]_{q}  & & S_2 \\
 & S_3 \ar[ru]_{p} &  } \] 
 where $q:S_1 \to S_3$ is surjective 
 and $p:S_3 \to S_2$ is induced by an inclusion of subgroups.  If we 
 further assume that the group $M_2$ in the presentation of 
 $S_2 = S_{G_2,\Gamma_2,M_2;\mathcal{F}}$ is compact, then $p$ has compact fibers.
\end{proposition}

\begin{proof} 
Take $S_i = S_{G_i,\Gamma_i,M_i;\mathcal{F}_i}$
for $1 \leq i \leq 2$ and let 
$\varphi:G_1 \to G_2$ be a definable homomorphism
and $a \in G_2$ so that $f:S_1 \to S_2$ is 
given by $\Gamma_1 x M_1 \mapsto \Gamma_2 a \varphi(x) M_2$.  

Define $G_3 := \varphi(G_1) \leq G_2$, a definable
group.  Set $\Gamma_3 := \varphi(\Gamma_1)$,
$M_3 := \varphi(M_1)$, and $\mathcal{F}_3 :=
\overline{\varphi}(\mathcal{F}_1)$ where 
$\overline{\varphi}:G_1/M_1 \to G_3/M_3$ is 
the induced map.   Since 
$\overline{\varphi}$ is a definable open map, 
$\mathcal{F}_3$ is open and clearly 
$\mathcal{F}_3$ is a fundamental set of the 
action of $\Gamma_3$ on $G_3/M_3$.   Set $S_3 := S_{G_3,\Gamma_3,M_3;\mathcal{F}}$
and let $q:S_1 \to S_3$ be the map $\Gamma_1 x M_1 \mapsto \Gamma_3 \varphi(x) M_3$.  
The map $p:S_3 \to S_2$ is then given by $\Gamma_3 x M_3 \mapsto \Gamma_2 a x M_2$.  
The fibers of $p$ are contained in a finite union of sets definably homeomorphic to the 
homogeneous space $(M_2 \cap G_3)/M_3$, which is compact if $M_2$ is.  
\end{proof}

For the problems we consider in this paper we require that our definable quotient spaces come equipped with 
a complex structure and for the domain $D$ to arise as a subset of an algebraic variety.

\begin{definition}
\label{def:complex-quotient-space}
A \emph{definable complex quotient space} is a definable quotient space $S_{\Gamma,G,M;\cF}$  
together with the data of a real algebraic group $\bG$ and an algebraic subgroup $\bB \leq \bG_\CC$ 
of the base change of $\bG$ to $\CC$ 
for which $G = \bG(\RR)^+$ is the connected component of the identity in the 
real points of a real algebraic group $\bG$,
$M = \bB(\CC) \cap G$, and $D = G/M \subseteq (\bG/\bB)(\CC) =: \chD(\CC)$ is an open domain 
in the complex points of the algebraic variety $\chD$. 

A morphism $f:S_{\Gamma_1,G_1,M_1;\cF_1} \to S_{\Gamma_2,G_2,M_2;\cF_2}$ 
of definable complex quotient spaces is a morphism of 
definable quotient spaces for which the definable map of groups is given by 
a map of algebraic groups $\varphi:\bG_1 \to \bG_2$ for which 
$\varphi(\bB_1) \leq \bB_2$.

The class of definable complex quotient spaces with this notion of morphism forms a category
$\dcqs$.

\end{definition}

We leave it to the reader to check that the proofs of the  
basic closure properties for the category $\dqs$, such as the existence of 
a terminal object and closure under fiber products, go through for the category $\dcqs$.
In practice, the morphisms in $\dcqs$ we consider satisfy a stronger conclusion than what
Proposition~\ref{prop:factorization} gives.  That is, in practice, a map 
$f:S_{G_1,\Gamma_1,M_1;\mathcal{F}_1} \to S_{G_2,\Gamma_2,M_2;\mathcal{F}_2}$ in $\dcqs$ is given 
by a map of algebraic groups $\varphi:\bG_1 \to \bG_2$ (and an element $a \in G_2$) for which 
$M_1$ is a finite index subgroup of $\varphi^{-1} (M_2)$.  It then follows from the proof of 
Proposition~\ref{prop:factorization} that $f$ factors as $f = p \circ q$ where $q:S_1 \to S_3$ is 
a surjective $\dcqs$ morphism and $p:S_3 \to S_2$ is a $\dcqs$ morphism with finite fibers.

As we have defined definable complex quotient spaces, such a space $S_{G,\Gamma,M;\mathcal{F}}$ may
have singularities.  When we restrict to the case that the group $M$ is compact, then the 
singularities are at worst locally isomorphic to those coming from a quotient by a finite group.  
In our applications, we will consider only cases where these quotients may be desingularized 
by passing to a finite cover by another definable complex quotient space.

For some purposes we may wish to restrict to an even smaller category $\sS$.  We always assume about 
our given category $\sS$ of definable complex quotient spaces that it satisfies some 
basic closure properties.  Let us specify these with the following convention.

\begin{convention}
\label{conv:basic-conditions-S}
The category $\sS$ is a subcategory of $\dcqs$ satisfying the following conditions.
\begin{itemize}
    \item The one point space $\{ \ast \}$ is a terminal 
object of $\sS$.
    \item The category  $\sS$ is closed under fiber products.
    \item Every $\sS$-morphism $f:S_1 \to S_2$ 
factors as $f = p \circ q$ where $q:S_1 \to S_3$ is 
a surjective $\sS$-morphism and $p:S_3 \to S_2$ is an $\sS$-morphism with finite fibers.
\end{itemize}
\end{convention}

In Section~\ref{sec:density} we will impose an additional restriction on $\sS$.

\begin{definition}
\label{def:strongly-special}
If $f:S_1 \to S_2$ is an $\sS$-morphism, then 
the image $f(S_1)$ is 
called an \emph{$\sS$-special} subvariety of $S_2$.
\end{definition}
 
 In Definition~\ref{def:strongly-special} we refer to $f(S_1)$ as a 
 special \emph{subvariety}.  It is, in fact, always 
 a complex analytic subvariety of $S_2$.  Indeed, if we factor 
 $f = p \circ q$ as given by Convention~\ref{conv:basic-conditions-S}, then $p:S_3 \to S_2$ 
 is a finite, and hence, proper, map of complex analytic spaces so that by Remmert's proper mapping theorem
 its image $p(S_3) = f(S_1)$
 is a complex analytic subvariety of $S_2$.  Using this observation, we may 
 modify Definition~\ref{def:strongly-special} to require the morphism $f:S_1 \to S_2$ 
 witnessing that $f(S_1)$ is an $\sS$-special subvariety of $S_2$ to be finite.

\begin{definition}
\label{def:family-weakly-special}
An \emph{$\sS$-family 
of weakly special subvarieties of $S \in \sS$}
is given by a pair of $\sS$-morphisms 
\[  \xymatrix{ & S_1 \ar[ld]_\zeta \ar[rd]^\xi & \\ S & & S_2 }\]
for which $\zeta$ is a finite map over its image
and $\xi$ is surjective.  For each 
$b \in S_2$, the image $\zeta \xi^{-1} \{ b \}$ is 
a \emph{weakly $\sS$-special} subvariety of $S$.
\end{definition}

\begin{remark}
An $\sS$-special variety is weakly $\sS$-special 
as if $f:S_1 \to S$ expresses $f(S_1)$ as an 
$\sS$-special subvariety of $S$ with $f$ finite, then 
we can take $\zeta = f$, $S_2 = \{ \ast \}$, $\xi:S_1 \to S_2$ 
the unique map to $\{ \ast \}$.  If we take $\sS$ to be the category 
$\dcqs$, then the converse that every weakly $\sS$-special variety is 
actually $\sS$-special holds.

In another extreme, every 
singleton in $S \in \sS$ is a weakly special variety witnessed by 
$S_1 = S_2$ and $\xi = \zeta = \operatorname{id}_S$.  

\end{remark}

For our results it will be important that all 
$\sS$-weakly special varieties come from countably 
many $\sS$-families of weakly special subvarieties.   We isolate
it as an hypothesis and verify it in cases of 
interest.  

\begin{definition}
\label{def:few-prespecial}
We say that the weakly $\sS$-special subvarieties of $S \in \sS$ are \emph{well-parameterized} if there are countably many 
$\sS$-families of weakly special subvarieties of 
$S$, \[  \xymatrix{ & S_{1,i} \ar[ld]_{\zeta_i} \ar[rd]^{\xi_i} & \\ S & & S_{2,i} } \]
for $i \in \mathbb{N}$ 
so that for every weakly 
$\sS$-special subvariety $S' \subseteq S$ of $S$ there 
is some $i \in \mathbb{N}$ and $c \in S_{2,i}$ so that
$S' = \zeta_{i} \xi_{i}^{-1} \{ c \}$. 
More generally, we say that the weakly $\sS$-special subvarieties
are well-parameterized if for every $S \in \sS$ the 
weakly $\sS$-special subvarieties of $S$ are well-parameterized.
\end{definition}

\begin{remark}
\label{remark:countable-S} When 
$\sS$ is itself countable, by which we mean that
there are only countably many objects in $\sS$ and 
the set of $\sS$-morphisms between any two such objects
is itself countable, then the weakly $\sS$-special subvarieties 
are well-parameterized.
\end{remark}

\begin{remark}
\label{remark:necessity-of-well-parameterized}    The well-paramaterization condition may fail in some 
cases.  From the theory of Douady spaces, we know that all complex analytic
subvarieties of a given compact definable complex quotient space $S$ are parameterized by countably many 
complex analytic families of analytic spaces.  However, when parameterizing 
weakly special varieties, the Douady universal family  need not arise as a weakly special family as defined in 
Definition~\ref{def:family-weakly-special}. Moreover, 
there  need not be natural 
parameterizations of the
weakly special varieties in the 
cases that $S$ is non-compact. 
The Hopf manifold construction  is be instructive here.
For example, consider $p$ and $q$ two 
multiplicatively independent complex numbers of modulus 
less than one, $X := \mathbb{C}^2 \smallsetminus \{ (0,0) \}$,  $\Gamma := \langle (p,q) \rangle$  is 
the subgroup of $\operatorname{Aut}(X)$ generated by 
the map $(x,y) \mapsto (px, qy)$, and 
$M := \Gamma \backslash X$. The Hopf manifold $M$ may 
be realized as a definable complex quotient space and is 
a compact complex manifold with a noncompact, connected group of automorphisms. One by one, these automorphisms 
define special subvarieties of $M \times M$, but they cannot
be parameterized by a family of weakly special varieties.
\end{remark}

Let us indicate now the key functional transcendence condition which may hold 
in a category $\sS$ of definable complex quotient spaces.

\begin{definition}
\label{def:Ax-Schanuel-condtion}
 Fix a category $\mathsf{S}$ of
definable complex quotient spaces. 
We say that   $f:X^\text{an} \to S' 
\subseteq S_{\Gamma,G,M;\cF} \in \sS$,
a 
definable complex analytic map
from the analytification of a complex algebraic variety $X$ to 
a weakly $\sS$-special variety $S' \subseteq S$ 
satisfies the
\emph{Ax-Schanuel condition} 
relative to $\mathsf{S}$ if whenever 
$k \in \mathbb{Z}_+$ is a positive integer,
and $(\gamma,\widetilde{\gamma}):\Delta^k \to X \times D$ is
a complex analytic map,
where \[ \Delta = \{ z \in \CC : \| z \| < 1 \} \text{ ,}\] with $\pi_\Gamma \circ \widetilde{\gamma}= f \circ \gamma$, then 
either \[ \trdeg{\CC}{\CC(\gamma,\widetilde{\gamma})} \geq \dim S' + \operatorname{rk}(d\widetilde{\gamma}) \] or 
$f(\gamma(\Delta^k))$ is contained in a proper weakly $\sS$-special subvariety of $S'$.
\end{definition}

Under the hypothesis that the weakly $\sS$-special subvarieties 
are well-parameterized, 
the Ax-Schanuel condition implies a uniform version of itself.   

Since our statement of this uniform version, expressed as 
Proposition~\ref{prop:uniform-Ax-Schanuel}  below,
is a bit dense, we take this opportunity to explain it with 
a few words.  Basically, what it says is that 
if we are given an $\sS$-family of weakly special
varieties and a family of algebraic varieties which 
might witness the failure of the transcendence
degree lower bound in the Ax-Schanuel property, 
then weakly special variety in the alternative
provided by the Ax-Schanuel property may be chosen from 
one of finitely many preassigned $\sS$-families of 
weakly special varieties.

Before proving Proposition~\ref{prop:uniform-Ax-Schanuel}
we require two lemmas.  The first describes families of
weakly special subvarieties algebraically.  The second allows
us to recast Ax-Schanuel in differential algebraic
terms.

\begin{lem}
\label{lem:algebraicity-weakly-special-family}
Let $f:X^\text{an} \to S_{\Gamma,G,M;\cF} =: S \in \sS$ 
be a definable complex analytic map from a complex 
algebraic variety to a definable complex quotient space in $\sS$. Let 
\[  \xymatrix{ & S_1 \ar[ld]_\zeta \ar[rd]^\xi & \\ S & & S_2 }\]
be an $\sS$-family of weakly special subvarieties.  Then there are algebraically 
constructible sets $B$ and $T \subseteq X \times B$ so that 
the set of fibers $\{ T_b : b \in B(\mathbb{C}) \}$
is equal to $\{ f^{-1} \zeta\xi^{-1} \{ c \} : c \in S_2 \}$. 
\end{lem}
\begin{proof}
By the Riemann Existence Theorem~\cite[Th\'{e}or\`{e}me 5.2]{SGA1XII},
there is an algebraic variety $X'$, a regular map of 
algebraic varieties $\zeta':X' \to X$ and an analytic map 
$f':(X')^\text{an} \to S_1$ realizing $(X')^\text{an}$ as
the fiber product $X^\text{an} \times_S S_1$.  The fiber equivalence relation 
\[E_{\xi \circ f'} := \{ (x,y) \in X' \times X' : \xi (f'(x)) = \xi(f'(y)) \} \] is analytic and definable, and hence algebraic by
the definable Chow theorem~\cite{PeSt}.    
The quotient $B := X'/E_{\xi \circ f'}$ may be 
realized 
within the category of constructible sets
as a constructible set.  Let us write 
$\nu:X' \to B$ for the quotient map.  
We may then take \[T := \{ (x,b) \in X \times B : (\exists x' \in X) \zeta'(x') = x 
\text{ and } \nu(x') = b \} \text{ .}\]
\end{proof}

\begin{definition}
\label{def:relatively-weakly-special}
Let $f:X^\text{an} \to S \in \sS$ be a definable 
complex analytic map from the complex algebraic 
variety $X$ to the 
definable complex quotient space $S$ in $\sS$.  We say that 
a subvariety $Y \subseteq X$ is \emph{relatively weakly $\sS$-special
of relative dimension at most $d$} if there is a weakly $\sS$-special
$S' \subseteq S$ of dimension at most $d$ for which $Y = f^{-1} S'$.
\end{definition}

 Note that in Definition~\ref{def:relatively-weakly-special},
because we allow for the possibility that
$f$ is not a finite map, it could happen that 
the dimension of $Y$ itself is greater than $d$. 
On the other hand, the intersection of $S'$ with 
$f(X)$ may even be empty!  Thus, the dimension of 
$Y$ could be less than $d$.

It follows from Lemma~\ref{lem:algebraicity-weakly-special-family}
that if the weakly $\sS$-special subvarieties are well-parameterized,
then for any definable complex analytic map
$f:X^\text{an} \to S \in \sS$ from a complex algebraic variety 
$X$ to some definable complex quotient space $S$ in $\sS$ that
we can recognize the 
pullbacks under $f$ of prespecials in the sense 
that for each number $d$ the collection of 
relatively weakly $\sS$-special subvarieties of 
dimension at most $d$ comprise a countable collection
of algebraic families of subvarieties of $X$.

\begin{definition}
\label{def:relatively-special-parameter}
Let $f:X^\text{an} \to S \in \sS$ be a definable 
complex analytic map from the complex algebraic variety $X$ to 
to a definable complex quotient space $S$ in $\sS$. 
Given any field $M$ over which $X$ and a countable collection
of families of relatively weakly $\sS$-special subvarieties
of $X$ including all such relatively weakly $\sS$-special 
subvarieties 
are defined, by an \emph{$M$-relatively weakly $\sS$-special 
variety of dimension at most $d$} we mean an $M$-variety
of the form $Y_b$ where $Y \subseteq X \times B$ is 
an algebraic family of weakly $\sS$-special subvarieties of 
dimension at most $d$ and $b \in B(M)$.
\end{definition}

Using the Seidenberg embedding theorem we may 
reformulate the Ax-Schanuel property in differential
algebraic terms.  To be completely honest, the embedding theorem as 
stated and proven by Seidenberg in~\cite{Sei1,Sei2} is not 
quite sufficient in that he starts with a finitely generated 
differential subfield $K \subseteq \mathcal{M}(U)$ of 
a differential field of meromorphic
functions on the open domain $U \subseteq \mathbb{C}^n$ and then 
shows that for any finitely generated differential field extension 
$L$ of $K$ at the cost of shrinking $U$ to some open subdomain 
$V \subseteq U$ we may embed $L$ into $\mathcal{M}(V)$ over the 
embedding of $K$.  For our purposes, we will need to 
start with a possibly countably generated differential field 
$K \subseteq \mathcal{M}(U)$.  The necessary extension of
embeddings theorem is a consequence of the the 
Cauchy-Kovalevskaya theorem~\cite{Kow} and appears as
Theorem~1 of~\cite{PPR}.  Iterating this construction 
countably many steps, we see that if $K \subseteq \mathcal{M}(U)$
is a countable differential subfield of the meromorphic 
functions on some open domain in $\mathbb{C}^n$ and 
$L$ is a countably generated differential field extension of 
$K$, then $L$ embeds into the differential field of germs of 
meromorphic functions at some point $x \in U$ over the 
embedding of $K$.

We recall the generalized Schwartzian and 
generalized logarithmic derivative constructions from~\cite{Sc}.  
Consider $S = S_{G,\Gamma,M;\mathcal{F}}$ a definable complex quotient space
and fix an integer $k$. 
From the action $\bG \curvearrowright \chD$ of the algebraic 
group $\bG$ on the quasiprojective algebraic variety $\chD$ 
and a positive integer $k$, there is a differentially 
constructible map $\widetilde{\chi}:\chD \to Z$ from $\chD$ to 
some algebraic variety $Z$ so that for any differential 
field $(L,\partial_1,\ldots,\partial_k)$ extending $\CC$ 
(where the derivations $\partial_i$ commute and vanish on
$\CC)$ we have that for $x, y \in \chD(L)$, 
\[ \widetilde{\chi}(x) = \widetilde{\chi}(y) \Longleftrightarrow (\exists g \in \bG(C) ) g x = y \]
where \[ C = \{ a \in L : \partial_i(a) = 0 \text{ for } 1 \leq i \leq k \} \] 
is the common constant field of $L$. In particular, if this common constant field is $\CC$, then 
we may express the quotient of $\chD$ by the $\bG(\CC)$ as the image of $\widetilde{\chi}$.  
When $f:X^\text{an} \to S$ is a definable, complex analytic map from the analytification of a
quasiprojective algebraic variety $X$ to $S$, then we may define a differentially analytically constructible
function $\chi:X \to Z$ by the rule that for any meromorphic $\gamma:U \to X$ (where 
$U \subseteq \CC^k$ is an open domain in $\CC^k$), $\chi(\gamma) := \widetilde{\chi} (\pi_\Gamma^{-1} (f(\gamma))$
where $\pi_\Gamma^{-1}$ is any branch of the inverse of $\pi_\Gamma$.  Theorem 3.12 of~\cite{Sc} shows that 
$\chi$ is actually differentially constructible.  (That theorem is stated in the case where 
$X^\text{an} = S$ and $f = \operatorname{id}_S$, but the proof goes through in the more general case.)

\begin{lem}
\label{lem:diff-alg-form-ax-schanuel}
 Fix $\mathsf{S}$ a category of 
definable complex quotient spaces.
Suppose  that the weakly $\sS$-special subvarieties are well-parameterized 
and that $f:X^\text{an} \to S' \subseteq S_{\Gamma,G,M;\cF} \in \sS$  
is a definable 
complex analytic map from the analytification of a complex algebraic variety $X$
to an $\mathsf{S}$-weakly special subvariety of a definable complex quotient space in $\sS$. 
Let $K$ be a countable 
subfield of $\mathbb{C}$ over which 
$X$ and a complete collection of algebraic 
families of relatively weakly $\sS$-special 
varieties are defined.  Let $k$ and $d$ be two 
positive integers.  Then for any 
differential field $(L,\delta_1, \ldots, \delta_k)$ 
with $k$-commuting derivations for which  
$K$ is a subfield of the constants $C$ of $L$
and $C$ is algebraically closed
and any $C$-relatively weakly $\sS$-special $Y$ of 
dimension at most $d$, if
$(\gamma,\widetilde{\gamma}) \in Y(L) \times \check{D}(L)$
satisfies $\chi(\gamma) = \widetilde{\chi}(\widetilde{\gamma})$,  
$\operatorname{rk}( \left( \delta_i \gamma \right)_{1 \leq i \leq k}) = k$ and
$\operatorname{tr.deg}_C C(\gamma,\widetilde{\gamma}) < 
d + k$, then there is a $C$-relatively weakly 
$\sS$-special $Z \subsetneq Y$ for which $\gamma \in Z(L)$.
\end{lem}

\begin{proof}
Consider $Y$ and $(\gamma,\widetilde{\gamma}) \in Y(L)$ as
in the statement of the lemma.   Let $M$ be a countable 
differential subfield of $L$ with an algebraically 
closed field of constants $C'$ containing $K$ and over 
which $Y$ and the point $(\gamma,\widetilde{\gamma})$ 
are defined.  By the embedding theorem, we may realize $M$ 
as a differential field of germs of meromorphic functions.
Let $S' \subseteq S$ be the $\sS$-weakly special
variety of dimension at most $d$ for which $Y = f^{-1} S'$.
By Ax-Schanuel applied to $\gamma$ and $\widetilde{\gamma}$
regarded as meromorphic functions, there is a proper 
weakly $\sS$-special subvariety $S'' \subsetneq S'$ with 
the image of $f \circ \gamma$ contained in $S''$.   
The algebraic variety $Z := f^{-1} S''$ is then relatively
$\mathbb{C}$-weakly special.  Let $M' := M^\text{dc}$ be the differential
closure of $M$ and $M'' := M(\mathbb{C})^\text{dc}$ be the differential
closure of the differential field generated over $M$ by $\mathbb{C}$.
In $M''$, $\gamma$ satisfies the condition that it belongs to a 
$\mathsf{C}$-relatively weakly special variety of relative dimension 
strictly less than $d$ where $\mathsf{C}$ is the constant field. 
As $M''$ is an elementary extension of $M'$, the same is true in 
$M'$.  Since the constant field of the differential closure is the 
algebraic of the constant field of the initial field, we 
see that $\gamma$ belongs to a $C$-relatively weakly special 
variety of relative dimension strictly less than $d$. 
\end{proof}

A uniform version of the Ax-Schanuel condition follows 
from Lemma~\ref{lem:algebraicity-weakly-special-family} using the 
compactness theorem.

\begin{proposition}
\label{prop:uniform-Ax-Schanuel}
 Let $f:X^\text{an} \to S' \subseteq S_{\Gamma,G,M;\cF} \in \sS$ be a definable 
complex analytic map from the analytification of a complex algebraic variety $X$
to a definable complex quotient space satisfying 
the same hypotheses as in Lemma~\ref{lem:diff-alg-form-ax-schanuel}. 

Given an $\sS$-family 
\[ \xymatrix{  & S_{1} \ar[ld]_{\zeta} \ar[rd]^{\xi} & \\ S & & S_{2} } \]
of weakly $\sS$-special subvarieties of $S$, a 
positive integer $k \in \mathbb{Z}_+$, and 
a family $Y \subseteq (X \times \check{D}) \times B$
of subvarieties of $X \times \check{D}$, then there 
are finitely many $\sS$-families of weakly 
special subvarieties 
\[ \xymatrix{  & S_{1,i} \ar[ld]_{\zeta_i} \ar[rd]^{\xi_i} & \\ S & & S_{2,i} } \]
for $1 \leq i \leq n$ so that for any pair 
of parameters $b \in B$ and $c \in S_2$ and 
analytic map $(\gamma,\widetilde{\gamma}):\Delta^k \to Y_b \subseteq 
X \times D$ with $\gamma(\Delta^k) \subseteq \zeta \xi^{-1} \{ c \} =: S'_c$, $f \circ \gamma = \pi_\Gamma \circ \widetilde{\gamma}$, $\operatorname{rk}(d \widetilde{\gamma}) = k$, and 
$\dim Y_b < k + \dim S_c'$, 
there is some $i \leq n$
and $d \in S_{2,i}$ for which 
$f \circ \gamma(\Delta^k) \subseteq \zeta_i \xi_i^{-1} \{ d \} \subsetneq S'_c$.
\end{proposition}

\begin{proof}
Apply the compactness theorem to Lemma~\ref{lem:algebraicity-weakly-special-family}.  
See the proofs of~\cite[Theorem 3.5]{As-wZP} or~\cite[Theorem 4.3]{Ki-exp} for 
details on how to formalize the compactness argument.
\end{proof}

\section{Density of special intersections}
\label{sec:density}

In this section we state and prove our general theorem that, when 
persistently likely, intersections with special varieties are dense.  

Throughout this section $\mathsf{S}$ is a category of definable complex 
quotient spaces satisfying our usual hypotheses from Convention~\ref{conv:basic-conditions-S}  and 
some further requirements. 
Let us specify   with the following convention the properties we require.

\begin{convention}
\label{conv:S-conditions}
The category $\mathsf{S}$ of definable complex quotient spaces satisfies the
following conditions.
\begin{itemize}
    \item $\mathsf{S}$ is closed under fiber products.
     \item The terminal definable 
    complex quotient (one point) space 
    $\{ \ast \}$ belongs to $\mathsf{S}$ 
    as do the unique maps 
    $S \to \{ \ast \}$ for $S \in \mathsf{S}$.
    \item If $f:S_1 \to S_2$ is an $\mathsf{S}$-morphism then there 
    are $\mathsf{S}$-morphisms $q:S_1 \to S_3$ and $p:S_3 \to S_2$ 
    so that $f = p \circ q$, $q$ is surjective, and $p$ has finite fibers.
    \item For every $S \in \mathsf{S}$ there is some smooth $S' \in \mathsf{S}$ and  
    a finite surjective $\mathsf{S}$-morphism $S' \to S$.
    \item  The $\mathsf{S}$-weakly special varieties are well-parameterized.
    \item  Every definable analytic map $f:X^\text{an} \to S \in \mathsf{S}$
    from the analytification of an algebraic variety to a definable complex quotient space in 
    $\mathsf{S}$ considered in this section 
    satisfies the Ax-Schanuel condition relative to $\mathsf{S}$ .
\end{itemize}
\end{convention}

With the following definition we specify what is meant by intersections being 
persistently likely.

\begin{definition}
\label{def:persistently-likely}
Let $S = S_{G,\Gamma,M;\mathcal{F}} \in \mathsf{S}$.  
We suppose that the inclusion
$S' = S_{G',\Gamma \cap G',M \cap H; \mathcal{F}'}  \hookrightarrow S$ is 
an $\mathsf{S}$-morphism where if $D = G/M$ and $D' = G'/(G' \cap M)$, then 
$\mathcal{F}' = \mathcal{F} \cap D'$.  
Suppose that $G' \leq G$ and $D' = G'/(M \cap G')$.  Let $f:X^\text{an} \to S$
be a definable complex analytic map.   We say that $X$ and $S'$ have 
\emph{likely intersection} if $\dim f(X) + \dim S' \geq \dim S$, where here, the dimension 
is the o-minimal dimension.   We say the intersection is 
\emph{persistently likely} if whenever $\zeta:S_1 \to S$ and $\xi:S_1 \to S_2$ are 
$\mathsf{S}$-morphisms with $\zeta$ finite and $\xi$ surjective, then 
$\dim \xi (\zeta^{-1} f(X)) + \dim \xi (\zeta^{-1} S') \geq 
\dim S_2$.
\end{definition}

Note that this definition does not actually 
require $f(X)\cap S'$ to be non-empty for 
the intersection to be likely. 

\begin{remark}
\label{rk:persistent-family}
In Definition~\ref{def:persistently-likely}, we have consider only one special 
subvariety $S'$ of $S$, but we really intend to consider many.  With the notation 
as in Definition~\ref{def:persistently-likely}, for any $g \in G$, the set 
$\pi_\Gamma(gD' \cap \mathcal{F}) \subseteq S$ is locally (away from $\pi_\Gamma(gD' 
\cap \partial(\mathcal{F}))$) a complex analytic subvariety of 
$S$ of dimension equal to that of $D'$.  Indeed, in several cases of interest for 
many choices of $g$ (where this may mean that the set of suitable $g$ is 
dense in $G$) the set $\pi_\Gamma(gD')$ is actually a special subvariety of $S$.  
Persistent likelihood of the intersection of $X$ with $S'$ is equivalent to the 
persistent likelihood of the intersection of $X$ with these $\pi_\Gamma(g D')$.
\end{remark}

With this definition in place we may now state our main theorem.

\begin{thm}
\label{thm:density-special}
Let $S  = S_{G,\Gamma,M;\mathcal{F}} \in \mathsf{S}$ and let
$f:X^\text{an} \to S$ be a definable complex analytic map from the irreducible quasi-projective
complex algebraic variety $X$ to $S$.  Suppose that $S' \subseteq S$ is a
special subvariety expressible as $\pi_\Gamma(D')$ where $D' \subseteq D = G/M$ is a 
homogeneous space in $D$. 
Suppose moreover that 
the intersection of $X$ with $S'$ is persistently likely. 
Let $U \subseteq X(\CC)$ be an open subset of the 
complex points of $X$.   Then the set 
$B := \{ g \in G : f(U) \cap \pi_\Gamma (gD' \cap \mathcal{F}) \neq \varnothing \}$ 
has nontrivial interior.  In particular, if the set 
$\{ g \in G : \pi_\Gamma(gD') \text{ is a special subvariety of } S \}$ is 
Euclidean dense in $G$, then the set of special intersections, 
\[f^{-1} \bigcup_{g \in G, \pi_\Gamma(g D') \text{ is special}} \pi_\Gamma(g D') \] 
is dense in $X$ for the Euclidean topology.
\end{thm}

\begin{proof}
We break the proof of Theorem~\ref{thm:density-special} into 
several claims.  The 
claims at the beginning 
of the proof are really 
just reductions 
permitting us to 
consider a simpler 
situation.  The 
main steps of the proof
begin with 
Claim~\ref{claim:dim-R}
in which we compute the 
dimension of the incidence
correspondence $R$.  We 
then use this computation 
to show that $B$ and 
$G$ have the same o-minimal 
dimension, so that 
$B$ has nontrivial interior
in $G$.

\begin{claim}
\label{claim:injective} We may assume that 
$f:X^\text{an} \to S$ is an embedding.  
\end{claim}
\begin{claimproof}
The equivalence relation $E_f := \{ (x,y) \in X \times X : 
f(x) = f(y) \}$ is a definable and complex analytic 
subset of the quasi-projective 
algebraic variety $X \times X$.  Hence, by the 
definable Chow theorem, $E_f$ is itself algebraic.   
Let $Y$ be a nonempty Zariski open subset of 
$X/E_f$, considered as constructible set.  Then $f$
induces an embedding $\overline{f}:Y \hookrightarrow S$ 
for which the image of $\overline{f}$ is dense in 
$f(X)$.  Shrinking $U$, we may assume that 
$f(U) \subseteq \overline{f}(Y)$ and then replacing 
$U$ by $U' := \overline{f}^{-1} f(U) \subseteq Y(\CC)$,
 we see that if the theorem holds for 
$\overline{f}:Y^\text{an} \to S$ and $U'$, then it also 
holds for $f:X^\text{an} \to S$ and $U$.
\end{claimproof}

With Claim~\ref{claim:injective} in place, from now on we will regard 
$X$ as a locally closed subvariety of $S$.    With the next claim we record 
the simple observations that it suffices to prove the theorem 
for any given open subset of $U$ in place of $U$ and that 
we take $U$ to be definable.

\begin{claim}
\label{claim:U-definable} If Theorem~\ref{thm:density-special}
holds for some nonempty open $V \subseteq U$ in place of $U$, 
then it holds as stated.  Moreover, we may assume that $U$ is definable.   
\end{claim}
\begin{claimproof}
The set $\{ g \in G : V \cap \pi(g D' \cap \mathcal{F}) \neq \varnothing \}$ is 
a subset of $\{ g \in G : U \cap \pi(g D' \cap \mathcal{F}) \neq \varnothing \}$.
Hence, if the former set has nonempty interior, so does the latter.  For the
``moreover'' clause apply the main body of the claim to the case that $V \subseteq U$ is 
a nonempty open ball.
\end{claimproof}

From now on we will take $U$ to be definable and will continue to refer to the open subset 
of $X$ under consideration as $U$ even after taking various steps to shrink it.

Another basic reduction we shall employ is that it suffices to prove the 
theorem for a finite cover of $S$.

\begin{claim}
\label{claim:thm-for-finite-cover} If $\rho:\widetilde{S} \to S$ is a finite surjective $\mathsf{S}$-morphism,
then we may find an instance of 
the statement of Theorem~\ref{thm:density-special}
with $\widetilde{S}$ in place of $S$ so that the 
truth of Theorem~\ref{thm:density-special} for 
$\widetilde{S}$ implies the result of $S$. 
\end{claim}
\begin{claimproof}
Filling the Cartesian square
\[ \xymatrix{ Y \ar[r]^{\overline{f}} \ar[d]_{\overline{\rho} } & \widetilde{S}  \ar[d]^{\rho} \\
X^\text{an} \ar[r]^{f} & S  } \]
we obtain a complex analytic space $Y := X^\text{an} 
\times_S \widetilde{S}$.   Since 
$\overline{\rho}:Y \to X^\text{an}$ is finite, 
$Y$ is itself the analytification of an algebraic 
variety.  Let $\widetilde{X}$ be a component of this 
algebraic variety and then let 
$\widetilde{U} := \overline{\rho}^{-1} U$.  

The map $\rho:\widetilde{S} \to S$ comes from a 
homomorphism of algebraic groups 
$\varphi:\widetilde{\mathbf{G}} \to \mathbf{G}$ 
and some element $a \in G$ where 
$\widetilde{S} = S_{\widetilde{G},\widetilde{\Gamma},\widetilde{M};\widetilde{\mathcal{F}}}$.  This 
map induces a map 
$\widehat{\rho}:\widetilde{G}/\widetilde{M} = : \widetilde{D} \to D$.  Let $\widetilde{D}'$ be a 
component of $\widehat{\rho}^{-1} D'$.   If we 
succeed in showing that $\{ g \in G' : 
\overline{f}(\widetilde{U}) \cap \pi_{\widetilde{\Gamma}}
(g \widetilde{D}' \cap \widetilde{\mathcal{F}} \}$ 
contains some nonempty open set $V$, then 
$a \varphi(V)$ would be a nonempty open subset of 
$\{ g \in G : f(U) \cap \pi_\Gamma (gD' \cap 
\mathcal{F}) \}$, as required.
\end{claimproof}

Let us record a useful consequence of 
Claim~\ref{claim:thm-for-finite-cover}.

\begin{claim}
\label{claim:smooth}  We may assume that 
$S$ is smooth.
\end{claim}
\begin{claimproof}
By Convention~\ref{conv:S-conditions}, we may 
find a finite and 
surjective $\mathsf{S}$-morphism $\widetilde{S} \to S$
with $\widetilde{S}$ smooth.  By Claim~\ref{claim:thm-for-finite-cover},  if we know the theorem for 
$\widetilde{S}$, then we may deduce it for $S$.
\end{claimproof}

Another useful consequence of Claim~\ref{claim:thm-for-finite-cover}
is that we may assume that 
$f(X)$ is dense in $S$ with respect to the weakly special 
topology.

\begin{claim}
\label{claim:dense-weakly-special}
We may assume that there is no proper weakly special 
variety $S'' \subsetneq S$ with $f(X) \subseteq S''$.
\end{claim}
\begin{claimproof}
Let us prove Theorem~\ref{thm:density-special} by 
induction on the dimension of $S$.  If $f(X) \subseteq S'' \subsetneq S$
where $S''$ is a weakly special variety, then 
we could find a 
finite $\mathsf{S}$-morphism $\zeta:S_1 \to S$, a surjective
$\mathsf{S}$-morphism $\xi:S_1 \to S_2$, and a point $b \in S_2$ so that 
$f(X) \subseteq \zeta(\xi^{-1} \{ b \})$.  
By Claim~\ref{claim:thm-for-finite-cover}, we may assume that 
$S_1 = S$ and $\zeta = \operatorname{id}_S$.  That is, 
$f(X) \subseteq \xi^{-1} \{ b \} =: S_b \subsetneq S$.   
By Convention~\ref{conv:S-conditions}, the weakly special variety 
$S_b$, being the fiber product of $S$ with the the one-point space
$\{ \ast \}$ over $S_2$, is definably isomorphic to a 
space in $\mathsf{S}$.  By induction on dimension, 
Theorem~\ref{thm:density-special} already holds for $S_b$. 
\end{claimproof}

By our hypothesis that the intersection between 
$X$ and $S'$ is persistently likely, it is, in 
particular, likely.  Since we have
reduced to the case that $f:X^\text{an} \to S$ is 
an embedding by Claim~\ref{claim:injective}, 
we may express the likeliness of this intersection by
an equation \[ \dim_{\CC} (X) + \dim_{\CC} (S') = 
\dim_{\CC}(S) + k \]
for some nonnegative integer $k$.   Notice that we
have expressed this equality with dimensions as
complex analytic spaces.  Later, when we write 
``$\dim$'' without qualification we mean the 
o-minimal dimension, for which we would have 
\[ \dim X + \dim S' = \dim S + 2k \text{ .} \]

By the uniform Ax-Schanuel condition, which holds in 
$\mathsf{S}$ by Convention~\ref{conv:S-conditions} and 
Proposition~\ref{prop:uniform-Ax-Schanuel}, 
there is a finite list of families
of weakly special varieties 
\[  \xymatrix{ & S_{1,i} \ar[ld]_{\zeta_i} \ar[rd]^{\xi_{i}} & \\ S & & S_{2,i}} \] 
for $1 \leq i \leq n$ where $\zeta_i:S_{1,i} \to S$
is a finite $\mathsf{S}$-morphism and 
$\xi_i:S_{1,i} \to S_{2,i}$ is a surjective 
$\mathsf{S}$-morphism and if $\ell$ is a natural 
number with 
$(\gamma,\widetilde{\gamma}):\Delta^{\ell} 
\to U \times D$ is complex analytic with 
$\pi_\Gamma \circ \widetilde{\gamma} = g \circ \gamma$,
$\operatorname{rk}(d \gamma) = \ell > k$, and
$\widetilde{\gamma}(\Delta^\ell) \subseteq g D'$ for 
some $g \in G$, then for some 
$i \leq n$ and $b \in S_{2,i}$ we have 
$\widetilde{\gamma}(\Delta^\ell) \subseteq 
\zeta_i (\xi_i^{-1} \{ b \} ) \subsetneq S$.  

\begin{claim}
We may assume that $\zeta_i:S_{1,i} \to S$ is 
surjective for each $i \leq n$.
\end{claim}
\begin{claimproof}
By Claim~\ref{claim:dense-weakly-special}, we have reduced to the case
that $f(X)$ is not contained in any 
proper weakly 
special subvariety of $S$.  For any $i \leq n$ with $\zeta_i(S_{1,i}) \neq S$, 
we would thus have that $\zeta_i(S_{1,i}) \cap f(X)$ is a proper 
complex analytic subvariety of $f(X)$.  Thus, we may shrink $U$ so that for 
such an $i$ we have $\zeta_i (S_{1,i}) \cap U = \varnothing$.  We will thus
never encounter weakly special varieties of the 
form $\zeta_i(\xi^{-1} \{ b \})$ with $\gamma(\Delta^\ell) \subseteq 
\zeta_i(\xi^{-1} \{ b \}) \cap U$.  Thus, we may omit these families 
of weakly special varieties from our list.
\end{claimproof}

We may adjust our family of 
weakly special varieties to remember
only the maps $\xi_i:S_{1,i} \to S_{2,i}$.

\begin{claim}
\label{claim:identity-weakly-special}
We may assume that $S_{1,i} = S$  and 
$\zeta_i:S_{1,i} \to S$ is the identity 
map $\operatorname{id}_S:S \to S$.
\end{claim}
\begin{claimproof}
Work by induction on $n$.  In the inductive
case of $n+1$, apply 
Claim~\ref{claim:thm-for-finite-cover}
to replace $S$ by $S_{n+1,1}$.  We then need
to replace $S_{i,1}$ for $i \leq n$ by 
$S_{i,1} \times_S S_{n+1,i}$.  Conclude 
by induction.
\end{claimproof}

We shrink $U$ once again to ensure that all of the fibers of 
$\xi_i$ have the same dimension when restricted to $U$.

\begin{claim}
\label{claim:constant-fiber} We may shrink $U$ to a smaller
nonempty open set so that for all $i \leq n$ there is some 
number $d_i \in \mathbb{N}$ so that for all $u \in U$ we have 
$\dim (\xi_i^{-1} \{ \xi_i(u)\} \cap f(U) ) = d_i$.
\end{claim}
\begin{claimproof}
For each $i \leq n$ and each natural number $j \leq \dim U$, 
let \[ F_{i,j} := \{ u \in U : \dim_u (f^{-1} \xi_i^{-1} \{ \xi_i(f(u)) \})\} = j   \text{ .}\]

Here $\dim_u ( ~ )$ refers to the o-minimal dimension at 
$u$.

The definable set $U$ is the finite disjoint union of the 
definable sets 
\[ \bigcap_{i=1}^n F_{i,d_i} \]
as $(d_1, \ldots, d_n)$ ranges through $[0,\dim(U)]^n$.  
We may cell decompose $U$ subjacent to these definable sets. 
Let $V$ be an open cell in this cell decomposition.  Then for 
some sequence $(d_1, \ldots, d_n)$ we have 
$V \subseteq  \bigcap_{i=1}^n F_{i,d_i}$.   Because $V$ is an open 
cell, for each $u \in V$, we have 
\begin{eqnarray*}
\dim_u (f^{-1} \xi_i^{-1} \{ \xi(f(u)) \}) &=& \dim_u (V \cap f^{-1} \xi_i^{-1} \{ \xi(f(u)) \})\\
&=& \dim (V \cap f^{-1} \xi_i^{-1}
\{ \xi_i (f(u)) \} ) \text{ .}
\end{eqnarray*}

Apply Claim~\ref{claim:U-definable} to conclude.
\end{claimproof}

Consider now the following incidence correspondence.

\[ R := \{ (u,g) \in U \times G : f(u) \in \pi_\Gamma(gD' \cap 
\mathcal{F} ) \}\]

Note that $R$ is definable.

\begin{claim}
\label{claim:dim-R}  We have \[ \dim(R) = \dim(G) + 2k \text{ .}\]
\end{claim}

\begin{claimproof}
Fix the base point $\ast \in D' = G'/M'$
corresponding to $M' = M \cap G'$ in the coset space.
For $u \in U$, let 
$\widetilde{u} \in \mathcal{F}$ with $\pi_\Gamma(\widetilde{u}) = u$.
Let $g_0 \in G$ with $g_0 \ast = \widetilde{u}$.  We will check that 
$R_u := \{ g \in G : (u,g) \in R \}$ is a homogenous 
space for $M \times G'$ with fibers isomorphic to $G' \cap M$.
Indeed, if $h \in G'$ and 
$m \in M$, we have $\widetilde{u} = g_0 m h h^{-1} \ast$, 
demonstrating that $f(u) \in \pi_{\Gamma}((g_0 m h) D' \cap 
\mathcal{F})$.  That is, $g_0 m h \in R_u$.
On the other hand, if $g \in R_u$, then we can find some $h$ 
so that $g_0 \ast = \widetilde{u} = gh^{-1} \ast$.  That is, 
$m := g_0^{-1} gh^{-1} \in M$, the stabilizer of $\ast$ in $G$. 
That is, $g = g_0 m h \in g_0 M G'$.     We compute that
for $g_1, g_2 \in G'$ and $m_1, m_2 \in M$, we have 
$g_0 g_1 m_1 = g_0 g_2 m_2$ only if $g_2^{-1} g_1 = m_2 m_1^{-1} =: h
\in G' \cap M = M'$.   

Using the fiber dimension theorem, since 
all fibers over $U$ have the same dimension, 
$\dim(M \times G') - \dim (M \cap G')$, we now compute that
\begin{eqnarray*}
\dim R & = & \dim f(U) + \dim R_u \text{ \small for any $u \in U$} \\
    & = & \dim U + \dim(M \times G') - \dim (M \cap G') \\
    & = & \dim X + \dim M + (\dim G' - \dim (M \cap G')) \\
    & = & \dim X + \dim M + \dim S' \\
    & = & \dim X + (\dim G - \dim S) + \dim S' \\
    & = & \dim G + 2k \text{ .}
\end{eqnarray*}
\end{claimproof}

Abusing notation somewhat, for $g \in G$ we will also write 
$R_g$ for the fiber $\{ u \in U : (u,g) \in R \}$.    Note that
$R_g$ is definably, complex analytically isomorphic to 
$f(U) \cap \pi_\Gamma(g D' \cap \mathcal{F}')$ which is a 
locally closed complex analytic subset of $S$.  It follows that the 
o-minimal dimension of $R_g$ is always even.

For each $i \leq \dim X$, let us define 
\[ B_i := \{ g \in G : \dim R_g = i \} \text{ .}\]

\begin{claim}
\label{claim:lower-bound}  For $i < 2k$, we have 
$B_i = \varnothing$.
\end{claim}
\begin{claimproof}
We have reduced through Claim~\ref{claim:smooth} to the case that 
$S$ is smooth.  Hence, each component of 
$f(U) \cap \pi_\Gamma(g D' \cap \mathcal{F}')$ has complex dimension 
at least $\dim U + \dim D' - \dim S = k$.
\end{claimproof}

The set $B$ of the statement of the theorem may be expressed as 
\[ B = \bigcup_{i=0}^{\dim U} B_i \text{ .}\]

By Claim~\ref{claim:lower-bound}, we actually have
\[ B = \bigcup_{i=2k}^{\dim U} B_i \text{ .}\]

With the next claim we show that (again by shrinking $U$) we 
may arrange that $B = B_{2k}$.

\begin{claim}
\label{claim:upper-bound}
Possibly after shrinking $U$, we have $B_i = \varnothing$ for $i > 2k$.
\end{claim}
\begin{claimproof}
Suppose that $\ell > k$ and $g \in B_{2\ell}$. Then the complex 
analytic set $f(U) \cap \pi_\Gamma(gD' \cap \mathcal{F})$ has 
a component $L$ of complex dimension $\ell$.  
Let $(\gamma,\widetilde{\gamma}):\Delta^\ell \to U \times g D'$ be a 
complex analytic map with $\operatorname{rk}(d \gamma) = \ell$ and
$\pi_\Gamma \circ \widetilde{\gamma} = f \circ \gamma$.   
By our choice of the witnesses to the Ax-Schanuel 
property for $\mathsf{S}$, for some $i \leq n$ and $b \in S_{2,i}$
we have \[ f \circ \gamma (\Delta^\ell) \subseteq \xi_i^{-1} \{ b \} =: 
S_b \subsetneq S \text{ .}\]  

By our reduction from Claim~\ref{claim:constant-fiber} and the 
fiber dimension theorem, we have

\[\dim X = \dim U = \dim f(U) = \dim \xi_i f(U) + \dim (f(U) \cap S_b)  \text{ .}\]

Moreover, by the homogeneity of $S'$, we also have 

\[ \dim S' = \dim \xi_i(S') + \dim (S' \cap S_b) \]
 and, of course, 

 \[ \dim S = \dim S_{2,i} + \dim S_b  \text{ .}\]

By our hypothesis of persistent intersection, we have

\[ \dim_\CC \xi_i f(U) + \dim_\CC \xi_i S' = \dim S_{2,i} + k'  \]
for some $k' \geq 0$.

Written in terms of o-minimal dimension this says

\[ \dim \xi f(U) + \dim \xi_i S' = \dim S_{2,i} + 2 k' \text{ .}\]

Combining this equalities, we compute that 

\[ \dim (f(U) \cap S_b) + \dim (S' \cap S_b) = \dim S_b + 2k - 2k' \text{ .} \]

Since, $k' \geq 0$, this means that the expected (complex) dimension 
of a component of $f(U) \cap S_b \cap S_b$ is at most $k$, but 
$L$ is such a component of complex dimension greater than $k$.  
That is, $L$ is an atypical component of the intersection inside
$S_c \subsetneq S$.   Applying uniform Ax-Schanuel again, we 
may extend the family of weakly special varieties 
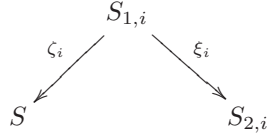
\[ \xymatrix{ & S_{1,i} \ar[ld]_{\zeta_i} \ar[rd]^{\xi_i} & \\
S &  & S_{2,i} } \] 

for $n+1 \leq i \leq n_2$ so that each such 
atypical component will satisfy 
$L \subseteq \zeta_i (\xi_i^{-1} \{ b_2 \} ) \subsetneq S_b \subsetneq S$ for some $i \leq n_2$ and $b_2 \in S_{2,i}$.

Repeating the reductions of the earlier claims and this extension of the 
list of weakly special witnesses to Ax-Schanuel $\dim S + 1$ times, 
we reach a contradiction to the hypothesis that $R_i$ is nonempty for 
some $i > 2k$.
\end{claimproof}

Thus, $B = B_{2k}$.  So we have 
\[ \dim G + 2k  = \dim  R = 
\dim B_{2k} + 2k = \dim B + 2k \text{ .}\]

Subtracting $2k$ from both sides, we conclude that 
$\dim B = \dim G$.  Hence, by cell decomposition, $B$ contains 
an open subset of $G$.
\end{proof}

 \section{Applications}
 \label{sec:applications}
 
 In this section we illustrate Theorem~\ref{thm:density-special} by considering 
 various situations in which it applies.
 
 \subsection{Arithmetic quotients}
 
 Our formalism is derived from that of Bakker, Klingler, and Tsimerman in~\cite{BKT} for the 
 study of arithmetic quotients.  They consider definable complex quotient spaces 
 $S_{G,\Gamma,M;\mathcal{F}}$ in which the algebraic group $\bG$ is a semisimple 
 $\mathbb{Q}$-algebraic group, $\Gamma$ is arithmetic (so commensurable with 
 $\bG(\mathbb{Z})$ for some / any choice of an integral model for $\bG$), and $M$ is
 compact.  They often require $\Gamma$ to be neat; we will return to that issue in a moment. 
 The definable fundamental domain $\mathcal{F}$ is not chosen sufficiently carefully in~\cite{BKT}, 
 an issue that was then addressed and fixed in \cite{BKT-erratum}.
 A similar issue is addressed in~\cite{OrSc} in that one needs to take $\mathcal{F}$ to be 
 constructed from a Siegel set associated to a maximal compact subgroup of $G$ containing $M$.  
 
 If we drop the neatness requirement on $\Gamma$, then an arithmetic quotient need not be smooth, but 
 because every arithmetic group has a neat subgroup of finite index, for any arithmetic quotient 
 $S$ we may find a smooth arithmetic quotient $\widetilde{S}$ and a finite surjective map of 
 arithmetic quotients $\widetilde{S} \to S$.   
 
 The one point space is clearly a terminal object in the category of arithmetic quotients 
 and the pullback construction of Proposition~\ref{prop:fiber-product} specializes to the 
 category of arithmetic quotients.   Since there are only countably many arithmetic quotients all told
 and at most countably many maps of algebraic groups between algebraic groups defined over 
 the rational numbers, it follows that the weakly special varieties are well-parameterized 
 within the category of arithmetic quotients.

The main theorem of~\cite{BaTs} is that period mappings associated to
 polarized variations of integral Hodge structures satisfy the 
 Ax-Schanuel condition with respect to $\mathsf{S}$.  It is an interesting open
 question whether every definable analytic map $f:X^\text{an} \to S$ 
 where $S$ is an arithmetic definable complex quotient space necessarily 
 satisfies the Ax-Schanuel condition with respect to $\mathsf{S}$.

Our last observation in verifying Convention~\ref{conv:S-conditions} and the hypotheses of 
Theorem~\ref{thm:density-special} for arithmetic quotients is that if 
$D' \subseteq D$ is a homogeneous space for which $\pi_\Gamma(D') \subseteq S$ is a special 
variety, then for every $g \in \mathbf{G}(\mathbb{Q}) \cap G$, $\pi_\Gamma(g D')$ is also special.
Thus the set of $g \in G$ for which $\pi_\Gamma(g D')$ is special is dense in $G$ for the Euclidean 
topology.

 Returning to the case where we know $f:X^\text{an} \to S$ to be a period mapping and 
 $f(X)$ to be Hodge generic in $S$, that is, not contained in any proper weakly special 
 subvariety, the union of $f^{-1} S'$ ranging over all proper special subvarieties $S' \subsetneq S$
 is called the Hodge locus.  In~\cite{Klingler-Otwinowska-HL}, a 
 dichotomy theorem is proven for a 
 modified form of the Hodge locus which 
 they call the Hodge locus of positive period dimension: either this locus is 
 Zariski dense in $X$ or it is itself a 
 proper algebraic subvariety of $X$.   In a very recent 
 preprint~\cite{KhUr}, tight conditions for the density of the Hodge loci 
 are established.

 Because the special subvarieties of $S$ come from $\mathbb{Q}$-semisimple
 algebraic subgroups of $\mathbf{G}$ and there are only finitely many such subgroups up to 
 $G = \mathbf{G}(\mathbb{R})^+$-conjugacy, all special subvarieties of $S$ come from finitely 
 many families of homogeneous spaces in the sense of Theorem~\ref{thm:density-special}.
 That is, we can find finitely many homogeneous spaces $D_1, \ldots, D_n \subseteq D$ so that 
 for any special subvariety $S' \subseteq S$ there is some $g \in G$ and $i \leq n$ with 
 $S' = \pi_\Gamma(g D_i)$.    Thus, if for some special subvariety $S' \subseteq S$ the intersection 
 of $X$ with $S'$ is persistently likely, then the Hodge locus is Euclidean dense in $X$.  In fact, we 
 may take $D_i$ so that $S' = \pi(g D_i)$ for some $g \in G$ and we see that the subset of the 
 Hodge locus of the form $f^{-1} \bigcup_{h \in G, \pi_\Gamma(h D_i) \text{ special }} \pi_\Gamma(h D_i)$
 is Euclidean dense in $X$.    In~\cite{BKU} a theorem of a similar flavor is proven.  They show that 
 if the \emph{typical} Hodge locus is nonempty, then it is analytically dense in $X$.  Here the typical 
 Hodge locus is the union of all components of $f^{-1} S'$ of expected dimension as $S'$ ranges through the 
 special subvarieties of $S$.  The proof in~\cite{BKU} uses some elements in common with ours.  Notably, 
 Ax-Schanuel plays a central role in both proofs.  To pass from a nonempty typical locus to one which is 
 dense, they argue an analysis of Lie algebras to find enough special varieties.  Such a technique is not 
 available to us in general as we must postulate the existence of special varieties of a given shape.
 On the other hand, such an argument does not immediately lend itself to a study of intersections with a 
 restricted class of special varieties.
 
Definability of the period mappings associated
to admissible, graded polarized,
variation of mixed Hodge structures has been 
established by Bakker, Brunebarbe, Klingler, and Tsimerman 
in~\cite{BBKT} and then Ax-Schanuel for these
maps was proven independently by Chiu~\cite{Chiu-ASwd}
and Gao and Klingler~\cite{GaKl}.  Indeed, Chiu
has established a stronger Ax-Schanuel theorem 
with derivatives for such period maps associated to 
variations of mixed Hodge structures~\cite{Chiu-AS-mixed}.  These results give the necessary 
ingredients to extend our result on the density of 
Hodge loci to variations of mixed Hodge 
structures.  We will return to the special case of 
universal abelian schemes over moduli spaces in
Section~\ref{sec:torsion-family} to draw a conclusion
from the combination of 
our Theorem~\ref{thm:density-special} and 
Ax-Schanuel in the context of mixed Shimura 
varieties.

 \subsection{Modular varieties}
 
 For the sake of illustration, let us consider a very special case of Theorem~\ref{thm:density-special}.
 We take $S = \mathbb{A}^n = X_0(1)^n$.  That is, $S$ is affine $n$-space (for some positive 
 integer $n$) regarded as the coarse moduli space of $n$-tuples of elliptic curves.  We may see 
 $S$ as an arithmetic quotient space, taking $\mathbf{G} = \operatorname{PGL}_2^n$, 
 $\Gamma = \operatorname{PGL}_2^n(\mathbb{Z})$, and 
 the homogeneous space $D$ may be identified with 
 $\mathfrak{h}^n$ where $\mathfrak{h} = \{ z \in \CC : \operatorname{Im}(z) > 0 \}$ is the upper half plane.  
 
 If $\sigma = \langle \sigma_1, \ldots, \sigma_m \rangle$ 
 is a finite sequence taking values in $\{ 1, \ldots, n \}$ we may define
 $\pi_\sigma:S \to \mathbb{A}^m$ by $(x_1,\ldots,x_n) \mapsto 
 (x_{\sigma_1},\ldots,x_{\sigma_n})$.   For 
 $J \subseteq \{1, \ldots, n \}$ we list the elements of 
 $J$ in order as $J = \{ j_1 < j_2 < \ldots < j_m \}$ and
 write $\pi_J$ for $\pi_{\langle j_1, \ldots, j_m \rangle}$.
 For a singleton $J = \{ j \}$, we just write $\pi_j$ for 
 $\pi_J$.
 
 For $S' \subseteq S$ a special subvariety 
 of $S$ provided that for each $i \leq n$ the projection map $\pi_i:S \to \mathbb{A}^1$ is dominant, then 
 $S'$ defines a partition $\Pi(S')$ of $\{ 1, \ldots, n \}$ by
 the rule that $i$ and $j$ lie in a common element of the
 partition if and only if $\dim \pi_{\langle i, j \rangle} S' = 1$.
 Given a partition $\Pi$ of $\{ 1, \ldots, n \}$ we say that 
 $S'$ is a special variety of type $\Pi$ if $\Pi(S') = \Pi$.  
 Let us observe that a special variety of type $\Pi$ has dimension 
 equal to $\# \Pi$. 

Fix a partition $\Pi$ with $\# \Pi = m$.  Let $D_\Pi' \subseteq D$
be the homogeneous subspace of $D$ defined by $\tau_i = \tau_j$ 
if and only there is some $\nu \in \Pi$ with $\{ i, j \} \subseteq
\nu$.   Then $\pi_\Gamma(D_\Pi') =: S'$ is the 
corresponding multi-diagonal subvariety of $S$ and 
is a special variety of type $\Pi$. Indeed, the special varieties 
of the form $\pi_\Gamma(gD_\Pi')$ as $g$ ranges through 
$\operatorname{PGL}_2(\mathbb{Q}) \cap G$ are exactly the 
special varieties of type $\Pi$.

For a partition $\Pi$ of $\{ 1, \ldots, n \}$ and 
subset $J \subseteq \{1, \ldots, n\}$ of $\{1, \ldots, n\}$, 
we define \[ \Pi \upharpoonright J := \{ 
\nu \cap J : \nu \in \Pi, \nu \cap J \neq \varnothing \} \]
to be the restriction of the partition $\Pi$ to $J$.   

It is easy to check that for any special variety $S' \subseteq S$
and subset $J \subseteq \{ 1, \ldots, n \}$, the 
projection $\pi_J(S')$ is a special variety and 
$\Pi( \pi_J(S')) = \Pi \upharpoonright J$.   With these
combinatorial preliminaries in place, we may state the 
specialization of Theorem~\ref{thm:density-special} to the 
case of $Y_0(1)^n$.

\begin{proposition}
\label{prop:special-intersection-j}
Let $n \geq 1$ be a positive integer and $\Pi$ a 
partition of $\{ 1, \ldots, n \}$.  If 
$X \subseteq \mathbb{A}^n$ is an irreducible complex 
algebraic subvariety of affine $n$-space, regarded as 
the coarse moduli space of $n$-tuples of elliptic curves,
and for every $J \subseteq \{ 1, \ldots, n \}$ we have
$\# \Pi \upharpoonright J  + \dim \pi_J(X) \geq \#J$, then
\[ X \cap \bigcup_{S' \subseteq \mathbb{A}^n 
\text{ special of type } \Pi } S' \]
is dense in $X$ for the Euclidean topology.
\end{proposition} 
\begin{proof}
Let us check that the intersection between $X$ 
and $S' := \pi_\Gamma(D_\Pi)$ is persistently likely.  
Let \[ \xymatrix{ & S_1 \ar[ld]_\zeta \ar[rd]^\xi & \\ S &  & S_2 }\] be a pair of surjective 
maps of arithmetic quotients with $\zeta$ 
finite.   The arithmetic quotients $S_1$ and $S_2$ will 
take the form $S_1 = S_{\operatorname{PGL}_2^n,\Gamma_1,M_1;
\mathcal{F}_1}$ and $S_2 = S_{\operatorname{PGL}_2^k,\Gamma_2,M_2;
\mathcal{F}_2}$ with $k \leq n$ where $\Gamma_j$ is an arithmetic
group in for $j = 1$ and $2$ and the corresponding 
homogeneous spaces are $\mathfrak{h}^n$ and $\mathfrak{h}^k$, 
respectively.  Since each $\operatorname{PGL}_2$ 
factor is simple, the maps of algebraic
groups corresponding $\zeta$ and $\xi$ are given by coordinate 
projections followed by an inner automorphism defined over 
$\QQ$.  That is, the map of groups corresponding to $\xi$ is 
given by $\langle g_1, \ldots, g_n \rangle \mapsto 
\langle g_{j_1}, \dots, g_{j_k} \rangle$ followed by 
an inner automorphism of $\operatorname{PGL}_2^k$ defined over
$\mathbb{Q}$ for some collection of $k$ distinct numbers
$j_1, \ldots, j_k$ between $1$ and $n$, and likewise for 
$\zeta$.  Let $J = \{ j_1, \ldots, j_k \}$, then 
permuting coordinates we see that this family of weakly special 
varieties fits into the commuting square 
\[  \xymatrix{ & S_1 \ar[ld]_\zeta \ar[rd]^\xi & \\ S \ar[rd]_{\pi_J}  &  & S_2 \ar[ld]^{\widetilde{\zeta}} \\ & \mathbb{A}^k &  } \]
where $\widetilde{\zeta}:S_2 \to \mathbb{A}^k$ is finite.  
We then have 
\begin{eqnarray*}
\dim \xi \zeta^{-1} X + \xi \zeta^{-1} S' & = & 
\dim \pi_J X + \dim \pi_J (S') \\
 & = & \dim \pi_J X + \# \Pi \upharpoonright J \\
 & \geq & k \\
 & = & \dim S_2 
\end{eqnarray*}

Since $\operatorname{PGL}_2^n(\mathbb{Q})^+$ is the
commensurator of $\operatorname{PGL}_2^n(\ZZ)$ and
is dense in $G = \operatorname{PGL}_2^n(\RR)^+$, 
the concluding ``in particular'' clause of Theorem~\ref{thm:density-special}
applies and we find that the intersections of $X$ with special varieties of 
type $\Pi$ is dense in $X$ in the Euclidean topology.
\end{proof} 

Instances of Proposition~\ref{prop:special-intersection-j} appear in the literature.
Habegger shows in~\cite[Theorem 1.2]{Hab} that if $X \subseteq \mathbb{A}^2$ is a curve defined
over the algebraic numbers, then there is a constants $c = c(X) > 0$ and $p_0(X) > 0$
so that for every prime number $p > p_0(X)$ there is an algebraic point 
$P \in X(\mathbb{Q}^\text{alg}) \cap Y_0(p) (\mathbb{Q}^\text{alg})$ 
with logarithmic height $h(P) \geq c \log(p)$ where here $Y_0(p)$ is the 
modular curve parametrizing the isomorphism classes of pairs of elliptic curves
$\langle E, E' \rangle$ for which there is an isogeny $E \to E'$ of degree $p$.
Habegger's result implies in particular that for $n = 2$ and $\Pi = \{ \{ 1, 2 \} \}$, 
if $X \subseteq \mathbb{A}^2_{\mathbb{Q}^\text{alg}}$ is an affine 
plane curve defined over the algebraic numbers, then the intersection of 
$X$ with the special varieties of type $\Pi$ is Zariski dense in $X$.  Using 
equidistribution results, this Zariski density could be upgraded to Euclidean density.

In the discussion after Remark 3.4.5 in~\cite{Za-book}, Zannier sketches an argument 
showing that if $X \subseteq \mathbb{A}^2$ is a rational affine plane curve, then the
intersections of $X$ with special curves of type $\Pi$, as in the previous paragraph, 
are dense in $X$ in the Euclidean topology.

 \subsection{Torsion in families of abelian varieties}
\label{sec:torsion-family}

 If $\pi:A \to B$ is an abelian scheme of relative 
 dimension $g$ over the irreducible 
 quasiprojective complex algebraic variety $B$ and
 $X \subseteq A$ is a quasi-section of $\pi$, by which we 
 mean that $\pi$ restricts to a generically finite map on $X$, 
 then under some mild nondegeneracy conditions, 
 we expect that the set  \[ \pi(X \cap A_{\text{tor}}) = \{ b \in B(\CC) :  (\exists n \in \ZZ_+) X_b \cap A_b(\CC)[n] \neq 
 \varnothing \} \]
of points on the base over which $X$ meets the torsion 
subgroup of the fiber is dense in $B$ if and only if 
$\dim B \geq g$.   Masser-Zannier prove in~\cite{MZ-ta} that
when $B = \mathbb{P}^1 \smallsetminus \{0, 1, \infty\}$ and 
$\pi:A \to B$ is the square of the 
Legendre family of elliptic curves defined in 
affine coordinates by $y_1^2 = x_1(x_1-1)(x_1-\lambda)$
and $y_2^2 = x_2(x_2 - 1)(x_2 - \lambda)$ where 
$\lambda$ ranges over $B$, and $X$ is the curve defined by
$x_1 = 2$ and $x_2 = 3$, then the set 
$\pi(X \cap A_{\text{tor}})$ is finite.  This theorem sparked
much work on torsion in families of abelian varieties 
culminating in a result announced by Gao and Habegger that, at 
least for such abelian schemes $\pi:A \to B$ defined over
$\QQ^\text{alg}$, if $X \subseteq A$ is an algebraic variety,
also defined over $\QQ^\text{alg}$ so that the group generated
by $X$ is Zariski dense in $A$ and $\pi(X \cap A_{\text{tor}})$
is Zariski dense in $B$, then $\dim X \geq g$.   
 
In the opposite direction, Andr\'{e}, Corvaja, and Zannier study
in~\cite{ACZ}  the problem of density of torsion 
through an analysis of the rank of the Betti map.
In an appendix to that paper written by Gao, it is shown
that if $\pi:A \to B$
is a principally polarized abelian scheme of relative dimension $g$ which has no non-trivial endomorphism (on any finite covering), and for which the image of $S$ in the moduli
space $\mathcal{A}_g$ of abelian varieties of dimenion $g$ 
itself has dimension at least $g$ and $X \subseteq A$ is 
the image of a section of $\pi$, then 
$\pi(X \cap A_{\text{tor}})$ is dense in $B$ in the 
Euclidean topology.  The proof of this result 
made use of the Ax-Schanuel theorem for pure Shimura
varieties and was subsequently upgraded.  See in 
particular Gao's work on the Ax-Schanuel theorem 
for the universal abelian variety~\cite{Gao-AS}  and on
the Betti map in~\cite{Gao-rank-Betti,Gao-rank-Betti-er}.

Gao's main theorem, Theorem 1.1, in~\cite{Gao-rank-Betti} may be seen as a geometric elaboration 
of what Theorem~\ref{thm:density-special} means 
for the density of torsion.  Gao considers 
an abelian scheme $\pi:A \to B$ of relative
dimension $g$ over a quasiprojective complex
algebraic variety $B$ and a closed irreducible
subvarierty $X \subseteq A$ and then establishes
the conditions under which the generic rank of the
Betti map restricted to $X$ may be smaller than 
expected.  It is noted with~\cite[Proposition 2.2.1]{ACZ} that density of the torsion in $X$ 
follows from the Betti map,  generically, 
having rank $2g$ on $X$.   Thus, the converse
of Gao's condition gives a criterion for when 
the torsion is dense. 

In more detail, taking finite covers if necessary,
one may pass from the problem 
of density of torsion in $X$ as a subvariety of $A$, 
to the density of torsion in $\widetilde{\iota}(X)$
in $\mathfrak{A}$ where $\mathfrak{A} \to \mathcal{A}$
is a universal abelian variety over 
a moduli space $\mathcal{A}$ of abelian varieties
of some fixed polarization type with some fixed
level structure and the Cartesian square 
\[ \xymatrix{A \ar[r]^{\widetilde{\iota}} \ar[d]_\pi 
& \mathfrak{A} \ar[d]^\pi \\ 
B \ar[r]^\iota & \mathcal{A} } \]
expresses $A \to B$ as coming from this universal 
family.   To ease notation, we replace
$B$ by $\iota(B)$ and $X$ by $\widetilde{X}$.  
Shrinking the moduli space, possibly taking 
covers, and moving to an abelian subscheme of
$\mathfrak{A}_B$, we may arrange that $X$ is 
not contained in any proper weakly special 
varieties.    At this point, 
Theorem~\ref{thm:density-special} says that 
$X \cap \mathfrak{A}_\text{tor}$ is dense 
in $X$ in the Euclidean topology if 
the intersection of $X$ with the zero section 
is persistently likely.  Gao's criterion 
expresses geometrically what persistent likelihood
means here: for any abelian subscheme 
$\mathfrak{A}'$ of $\mathfrak{A}_B$, if 
$p:\mathfrak{A}_B \to \mathfrak{A}_B / \mathfrak{A}'$
is the quotient map, then $\dim p(X)$ is at least
the relative dimension of 
$\mathfrak{A}_B / \mathfrak{A}'$ over $B$.

\bibliographystyle{siam}
\bibliography{li.bib}

\end{document}